\documentclass[12pt]{article}
\usepackage{amsmath,amssymb}

\newtheorem{theorem}{Theorem}[section]
\newtheorem{proposition}[theorem]{Proposition}
\newtheorem{corollary}[theorem]{Corollary}

\newtheorem{remark}[theorem]{Remark}

\newenvironment{proof}{\smallskip\par{\sc Proof.}\enspace}%
 {{\unskip\nobreak\hfil\penalty50\hskip2em
          \hbox{}\nobreak\hfil{\rule[-1pt]{5pt}{10pt}}
          \parfillskip=0pt\finalhyphendemerits=0
          \par\medskip}} 

\makeatletter
\def\section{\@startsection {section}{1}{\z@}{3.25ex plus 1ex minus
 .2ex}{1.5ex plus .2ex}{\large\bf}}
\def\subsection{\@startsection{subsection}{2}{\z@}{3.25ex plus 1ex minus
 .2ex}{1.5ex plus .2ex}{\normalsize\bf}}
\@addtoreset{equation}{section} 
\makeatother

\begin{document}

\begin{center}
\LARGE\textsf{A comparison theorem for stochastic differential
equations under a Novikov-type condition}
\end{center}

\vspace*{.3in}

\begin{center}
\sc
\large{Alberto Lanconelli}\\
\end{center}

\begin{center}
\sc
Dipartimento di Matematica \\
Universita' degli Studi di Bari\\
 Via E. Orabona, 4\\
 70125 Bari - Italy\\
E-Mail: {\it alberto.lanconelli@uniba.it}
\end{center}

\vspace*{.3in}

\begin{abstract}
We consider a system of stochastic differential equations driven by
a standard $n$-dimensional Brownian motion where the drift
coefficient satisfies a Novikov-type condition while the diffusion
coefficient is the identity matrix. We define a vector $Z$ of square
integrable stochastic processes with the following property: if the
filtration of the translated Brownian motion obtained from the
Girsanov transform coincides with the one of the driving noise then
$Z$ coincides with the unique strong solution of the equation;
otherwise the process $Z$ solves in the strong sense a related
stochastic differential inequality. This fact together with an
additional assumption will provide a comparison result similar to
well known theorems obtained in the presence of strong solutions.
\end{abstract}



\section{Introduction}
The present paper is devoted to the study of the system of
stochastic differential equations
\begin{eqnarray}\label{main SDE}
dX^x_t&=&b(X_t^x)+dB_t,\quad t\in ]0,T]\\
\nonumber X_0^x&=&x\in\mathbb{R}^n
\end{eqnarray}
where $b:\mathbb{R}^n\to\mathbb{R}^n$ is a measurable function,
$\{B_t\}_{0\leq t\leq T}$ is a standard $n$-dimensional Brownian
motion and $T\in\mathbb{R}$ is a fixed positive constant.\\
Equation (\ref{main SDE}) has attracted the attention of many
authors since it represents one of the simplest (the noise appears
additively) nontrivial stochastic perturbations of the ordinary
differential equation
\begin{eqnarray}\label{deterministic}
\frac{dX^x_t}{dt}&=&b(X_t^x),\quad t\in ]0,T]\\
\nonumber X_0^x&=&x\in\mathbb{R}^n.
\end{eqnarray}
One of the most fascinating features of equation (\ref{main SDE}) is
that it possesses a (strong) solution under very weak assumptions on
the drift function $b$ (assumptions that are not sufficient for the
existence of a solution to the deterministic equation
(\ref{deterministic})). It was in fact proved by Zvonkin \cite{Z}
that in the one dimensional case it is sufficient to require the
boundedness of $b$ in order to have a unique strong solution for the
equation (\ref{main SDE}). This result was then generalized to
several dimensions by Veretennikov \cite{V} who employed the
so-called Yamada-Watanabe principle (see e.g. \cite{KS}) and
techniques from the theory of partial differential equations.
Equation (\ref{main SDE}) was also considered by Krylov and
R\"ockner in \cite{KR}: in this paper they obtain strong solvability
of the
equation under an integrability condition of drift function $b$.\\
In \cite{MP} Meyer-Brandis and Proske proved that in the one
dimensional case, when the drift function is bounded, the unique
strong solution of (\ref{main SDE}) is even differentiable in the
sense of the Malliavin calculus (see e.g. \cite{Nualart}). Their
approach is based on a representation formula for solutions of
stochastic differential equations obtained in Lanconelli and Proske
\cite{LP}, white noise techniques and an approximation argument. We
also mention the recent paper by Nilssen \cite{Nilssen} where the
above mentioned Malliavin differentiability is obtained for
stochastic equations with a sub-linear drift.\\
We would like to mention the paper by Pilipenko \cite{P} who proved
the existence of a unique strong solution for the one dimensional
version of (\ref{main SDE}) with a bounded drift $b$ and with $B_t$
replaced by a stable symmetric Levy process. In \cite{DFPR} Da Prato
et al. obtained existence and uniqueness of a mild solution of an
infinite dimensional version of (\ref{main SDE}) with bounded drift.
Finally we remark that the above regularizing effect obtained with
the introduction of a noise into a deterministic differential
equation was also investigated for stochastic partial differential
equations in Flandoli et al. \cite{FGP} and Fedrizzi and Flandoli
\cite{FF}. Here the noise is added in a multiplicative way through
a Stratonovich integral.\\
In this paper we consider the stochastic differential equation
(\ref{main SDE}) with a drift function $b$ satisfying a Novikov-type
condition (see (\ref{novikov}) below). Under this assumption the
Girsanov theorem guarantees the existence of a weak solution for
that equation. This solution becomes the unique strong solution of
(\ref{main SDE}) if the shift utilized in the Girsanov theorem is
invertible (see the paper by \"Ust\"unel \cite{U} for an elegant
approach to stochastic equations via results on invertible shifts on
the Wiener space) or equivalently if the filtration of the
translated Brownian motion coincides with the filtration of the
original Brownian motion. We define a $n$-dimensional vector of
stochastic processes $Z$, characterized through a duality relation
(see (\ref{def of Z}) below) and we show that it coincides with the
unique strong solution of (\ref{main SDE}) when the above mentioned
filtrations agree. Observe that the process $Z$ is well defined
without the assumption of the invertibility of the shift or the
equality of the filtrations. We then ask ourselves the following
question: does the process $Z$ with the sole assumption
(\ref{novikov}) solve in the strong sense any equation related to
(\ref{main SDE})? Our main theorem states that the process $Z$
solves in the strong sense a stochastic differential inequality
where the drift function is the convex envelope of $b$, i.e. the
best convex minorant of $b$. This fact together with the local
Lipschitz-continuty of convex functions will provide a comparison
theorem for equation (\ref{main SDE}) where the role of the a priori
non-existing strong solution is played by the process $Z$.\\
(By comparison theorem we mean roughly speaking a result of the
following type: if the drift function of a stochastic differential
equation is greater than the drift function of another one then the
strong solution of the first equation is greater than the
strong solution of the second one).\\
The paper is organized as follows: In Section 2 after a quick
description of our framework we state the main result (Theorem
\ref{main theorem}) together with a corollary; Section 3 collects a
number of properties and useful remark about the process $Z$ that
will be utilized in Section 4 in proving the main results.

\section{Statement of the main results}

Let $(\Omega,\mathcal{F},\mathcal{P})$ be the classical Wiener space
associated with a standard $n$- dimensional Brownian motion over the
time interval $[0,T]$; more precisely, $\Omega$ coincides with the
space of continuous functions $\omega$ defined on the interval
$[0,T]$ with values on $\mathbb{R}^n$ and with
$\omega(0)=0\in\mathbb{R}^n$ ($T$ is a fixed positive number),
$\mathcal{F}$ is the Borel $\sigma$-algebra associated to the
topology induced by the sup-norm on $\Omega$ and $\mathcal{P}$
is the Wiener measure on $\Omega$.\\
We denote by $\{B_t\}_{0\leq t\leq T}$ the coordinate process
defined as $B_t(\cdot):\omega\in\Omega\mapsto
\omega(t)\in\mathbb{R}^n$, which turns out to be a standard
$n$-dimensional Brownian motion under the measure $\mathcal{P}$. We
also denote by $\{\mathcal{F}_t\}_{0\leq
t\leq T}$ its $\mathcal{P}$-augmented natural filtration.\\
Let $b:\mathbb{R}^n\to\mathbb{R}^n$ be a measurable function
satisfying the \emph{Novikov-type condition}
\begin{equation}\label{novikov}
E\Big[\exp\Big\{\int_0^T\Vert b(B_t+x)\Vert^2dt\Big\}\Big]<+\infty
\end{equation}
for all $x\in\mathbb{R}^n$. Here $E$ denotes the expectation on the
probability space $(\Omega,\mathcal{F},\mathcal{P})$ and
$\Vert\cdot\Vert$ stands for the $n$-dimensional Euclidian norm.
Observe that in the usual Novikov condition there is a factor
$\frac{1}{2}$ in front of the integral of formula (\ref{novikov}).
Therefore our condition is stronger than the classical one. The need
of this stronger condition will be evident in the proof of
Proposition \ref{adaptedness} below. Roughly speaking condition
(\ref{novikov}) guarantees that any process of the form
$\{b(B_t+x)+f(t)\}_{0\leq t\leq T}$, where $f$ is deterministic,
fulfils the usual Novikov condition.\\
According to the Girsanov theorem (see e.g. \cite{KS} page 190) the
process
\begin{eqnarray*}
B^b_t:=B_t-\int_0^tb(B_s)ds,\quad t\in [0,T],
\end{eqnarray*}
is a standard $n$-dimensional Brownian motion with respect to the
filtration $\{\mathcal{F}_t\}_{0\leq t\leq T}$ and the measure
\begin{eqnarray*}
d\mathcal{Q}:=\exp\Big\{\int_0^T\langle
b(B_s),dB_s\rangle-\frac{1}{2}\int_0^T\Vert
b(B_s)\Vert^2ds\Big\}d\mathcal{P}.
\end{eqnarray*}
(Here $\int_0^T\langle b(B_s),dB_s\rangle$ stands for
$\sum_{i=1}^n\int_0^Tb_i(B_s)dB_s^i$). We also denote by
$\{\mathcal{F}^b_t\}_{0\leq t\leq T}$ the
$\mathcal{P}$-augmented natural filtration of the process $B^b$.\\
We recall that condition (\ref{novikov}) guarantees that the process
\begin{eqnarray*}
t\in[0,T]\mapsto\exp\Big\{\int_0^t\langle
b(B_s),dB_s\rangle-\frac{1}{2}\int_0^t\Vert b(B_s)\Vert^2ds\Big\}
\end{eqnarray*}
is a continuous $(\{\mathcal{F}_t\}_{0\leq t\leq
T},\mathcal{P})$-martingale belonging to
$\mathcal{L}^q(\Omega,\mathcal{F},\mathcal{P})$ for all $q\in
[1,+\infty[$ (see \cite{KS} page 198).\\
In the sequel for an $n$-dimensional $\{\mathcal{F}_t\}_{0\leq t\leq
T}$-adapted process $\gamma=(\gamma^1,...,\gamma^n)$ satisfying
$\mathcal{P}(\sum_{i=1}^n\int_0^T|\gamma^i_s|^2ds<+\infty)=1$ the
following notation will be adopted
\begin{eqnarray*}
\mathcal{E}_t(\gamma):=\exp\Big\{\int_0^t\langle
\gamma_s,dB_s\rangle-\frac{1}{2}\int_0^t\Vert
\gamma_s\Vert^2ds\Big\}\quad t\in ]0,T[.
\end{eqnarray*}
We will simply write $\mathcal{E}(\gamma)$ when $t=T$ and $\mathcal{E}(b)$ when $\gamma_t=b(B_t)$ for any $t\in[0,T]$.\\
We now introduce the \emph{translation} (or \emph{shift})
\emph{operator}: for $k\in\mathbb{N}$, $\varphi\in
C_0^{\infty}(\mathbb{R}^{kn})$ and $t_1,...,t_k\in [0,T]$ define
\begin{eqnarray*}
\mathcal{T}_{-b}\big(\varphi(B_{t_1},...,B_{t_k})\big):=\varphi\Big(B_{t_1}-\int_0^{t_1}b(B_s)ds,...,B_{t_k}-\int_0^{t_k}b(B_s)ds\Big).
\end{eqnarray*}
This operator can be extended to
$\mathcal{L}^q(\Omega,\mathcal{F},\mathcal{P})$ for all $q\in
[1,+\infty]$. It is easy to see that under the condition
(\ref{novikov}), $\mathcal{T}_{-b}$ maps
$\mathcal{L}^q(\Omega,\mathcal{F},\mathcal{P})$
into $\cap_{r<q}\mathcal{L}^r(\Omega,\mathcal{F},\mathcal{P})$.
Note that in this notation we have $B_t^b=\mathcal{T}_{-b}B_t$.\\
If $h:\mathbb{R}^n\to\mathbb{R}$ is a measurable function we denote
by $\hat{h}$ the \emph{convex envelope} of $h$, i.e. $\hat{h}$ is
the greatest element of the set
\begin{eqnarray*}
\big\{g:\mathbb{R}^n\to\mathbb{R}\mbox{ convex and }g(x)\leq
h(x)\mbox{ for all }x\in\mathbb{R}^n\big\}.
\end{eqnarray*}
(Convex envelopes are well studied objects in optimization theory;
see the interesting paper by Oberman \cite{O} for a fully nonlinear
partial differential equation solved in the sense of
viscosity solutions by the convex envelope of a given function).\\
We conclude our framework saying that a function
$f:\mathbb{R}^n\to\mathbb{R}^n$ is \emph{quasi-monotonously
increasing} if for any $i\in\{1,...,n\}$ the inequality
\begin{eqnarray*}
f_i(x)\leq f_i(y)
\end{eqnarray*}
is satisfied for all $x,y\in\mathbb{R}^n$ such that $x_i=y_i$ and
$x_j\leq y_j$ if $j\neq i$. (This condition is well known in the
theory of systems of deterministic ordinary differential equations;
see for instance the discussion in Assing and Manthey \cite{AM}
and the references quoted there).\\
All the equalities and inequalities involving $n$-dimensional
vectors are understood in a componentwise manner.\\

\noindent We are now ready to state our main result.

\begin{theorem}\label{main theorem}\quad\\
Fix $x\in\mathbb{R}^n$ and let $b:\mathbb{R}^n\to\mathbb{R}^n$ be a
measurable function
fulfilling condition (\ref{novikov}).\\
For $i\in\{1,...,n\}$ and $t\in [0,T]$ define $Z_t^{x,i}$ to be the
unique element of $\mathcal{L}^2(\Omega,\mathcal{F},\mathcal{P})$
verifying for any
$Y\in\mathcal{L}^2(\Omega,\mathcal{F},\mathcal{P})$ the following
identity
\begin{eqnarray}\label{def of Z}
E[Z_t^{x,i}Y]&=&E[(B_t^i+x_i)\mathcal{E}(b)\mathcal{T}_{-b}Y].
\end{eqnarray}
Then:\\
i) The process $Z^x_t:=(Z_t^{x,1},...,Z_t^{x,n})$ is continuous and
$\{\mathcal{F}_t\}_{0\leq t\leq T}$-adapted.\\
ii) If $\mathcal{F}^b_t=\mathcal{F}_t$ for all $t\in [0,T]$, the
process $\{Z_t^x\}_{0\leq t\leq T}$ is the unique strong solution of
the system of stochastic differential equations
\begin{eqnarray}\label{SDE}
dX^x_t&=&b(X^x_t)dt+dB_t\quad t\in ]0,T]\\
\nonumber X^x_0&=&x
\end{eqnarray}
iii) If $\mathcal{F}^b_t\subset\mathcal{F}_t$ for some $t\in [0,T]$,
the process $\{Z_t^x\}_{0\leq t\leq T}$ solves the following system
of stochastic differential inequalities
\begin{eqnarray}\label{SDI}
dZ^x_t&\geq&\hat{b}(Z^x_t)dt+dB_t\quad t\in ]0,T]\\
\nonumber Z^x_0&=&x.
\end{eqnarray}
where $\hat{b}:=(\hat{b}_1,...,\hat{b}_n)$ and $\hat{b}_i$ is the
convex envelope of $b_i$ for each $i\in\{1,...,n\}$.\\
If in
addition $\hat{b}$ is quasi-monotonously increasing one has with
probability one that
\begin{eqnarray}\label{comparison}
Z_t^x\geq Y_t^x\mbox{ for all }t\in [0,\tau]
\end{eqnarray}
where $\{Y_t^x\}_{0\leq t\leq\tau}$ is the unique (possibly local)
strong solution of the stochastic differential equation
\begin{eqnarray}\label{maximal solution}
dY^x_t&=&\hat{b}(Y^x_t)dt+dB_t\quad t\in ]0,T]\\
\nonumber Y^x_0&=&x
\end{eqnarray}
and $\tau$ its explosion time.
\end{theorem}

\begin{corollary}\label{corollary}\quad\\
Fix $x\in\mathbb{R}^n$ and let $b:\mathbb{R}^n\to\mathbb{R}^n$ be a
measurable function fulfilling condition (\ref{novikov}). Assume in
addition that
\begin{eqnarray*}
b(x)\geq Ax+b\mbox{ for all }x\in\mathbb{R}^n
\end{eqnarray*}
where $A=(a_{ij})_{1\leq i,j\leq n}$ is a $n\times n$-matrix with
$a_{ij}\geq 0$ if $i\neq j$ and $b\in\mathbb{R}^n$.\\
Then the process $\{Z_t^x\}_{0\leq t\leq T}$ defined in (\ref{def of
Z}) verifies
\begin{eqnarray}\label{linear}
Z_t^x\geq
\Phi(t)\Big(x+\int_0^t\Phi^{-1}(s)bds+\int_0^t\Phi^{-1}(s)dB_s\Big)\mbox{
for all }t\in [0,T]
\end{eqnarray}
with probability one, where $t\mapsto\Phi(t)$ is the unique solution
of the following matrix differential equation
\begin{eqnarray*}
\frac{d}{dt}\Phi(t)=A\Phi(t),\quad \Phi(0)=I.
\end{eqnarray*}
($I$ denotes the $n\times n$ identity matrix).
\end{corollary}

\section{Preliminary results}

In this section we will only assume that
$b:\mathbb{R}^n\to\mathbb{R}^n$ is a measurable function satisfying
condition (\ref{novikov}).\\
Let $q>1$,
$X\in\mathcal{L}^{q+}(\Omega,\mathcal{F},\mathcal{P}):=\cup_{r>q}\mathcal{L}^{r}(\Omega,\mathcal{F},\mathcal{P})$
and consider the linear operator
\begin{eqnarray*}
Y\mapsto E[X\mathcal{E}(b)\mathcal{T}_{-b}Y].
\end{eqnarray*}
This map is continuous on
$\mathcal{L}^{q'}(\Omega,\mathcal{F},\mathcal{P})$, where $q'$
denotes the conjugate exponent of $q$. In fact,
\begin{eqnarray}\label{first estimate}
|E[X\mathcal{E}(b)\mathcal{T}_{-b}Y]|&\leq&(E[|X|^q\mathcal{E}(b)])^{\frac{1}{q}}(E[|\mathcal{T}_{-b}Y|^{q'}\mathcal{E}(b)])^{\frac{1}{q'}}\\
\nonumber&\leq&
(\Vert |X|^q\Vert_p\Vert\mathcal{E}(b)\Vert_{p'})^{\frac{1}{q}}\Vert Y\Vert_{q'}\\
\nonumber&\leq&\Vert X\Vert_{qp}\Vert\mathcal{E}(b)\Vert^{\frac{1}{q}}_{p'}\Vert Y\Vert_{q'}\\
\nonumber&=&C\Vert X\Vert_{qp}\Vert Y\Vert_{q'}
\end{eqnarray}
where we used the H\"{o}lder inequality twice and the Girsanov
theorem ($p$ and $p'$ are conjugate exponents with $p>1$). By the
Riesz representation theorem there exists a unique element $Z(X)$
(we stress only the dependence on $X$ and not on $b$ since this
function is arbitrary but fixed) belonging to
$\mathcal{L}^{q}(\Omega,\mathcal{F},\mathcal{P})$ such that
\begin{eqnarray}\label{riesz}
E[X\mathcal{E}(b)\mathcal{T}_{-b}Y]=E[Z(X)Y]
\end{eqnarray}
This equality will be crucial for the rest of the paper. The next
proposition explains how $Z(X)$ depends on $X$.
\begin{proposition}\quad\\
The operator
\begin{eqnarray*}
Z:\mathcal{L}^{q+}(\Omega,\mathcal{F},\mathcal{P})&\to&\mathcal{L}^{q}(\Omega,\mathcal{F},\mathcal{P})\\
X&\mapsto& Z(X)
\end{eqnarray*}
is linear and continuous.
\end{proposition}
\begin{proof}
The linearity follows immediately from (\ref{riesz}). Let us prove
the continuity:
\begin{eqnarray*}
\Vert Z(X)\Vert_q&=&\sup_{\Vert Y\Vert_{q'}\leq 1}|E[Z(X)Y]|\\
&=&\sup_{\Vert Y\Vert_{q'}\leq 1}|E[X\mathcal{E}(b)\mathcal{T}_{-b}Y]|\\
&\leq&\sup_{\Vert Y\Vert_{q'}\leq 1} C\Vert X\Vert_{q+}\Vert
Y\Vert_{q'}\\
&=&C\Vert X\Vert_{q+}
\end{eqnarray*}
where in the inequality we used the estimate obtained in (\ref{first
estimate}).
\end{proof}

\begin{proposition}\quad\\
Let $X\in\mathcal{L}^{q+}(\Omega,\mathcal{F},\mathcal{P})$. Then:\\
i) If $\varphi:\mathbb{R}\to\mathbb{R}$ is a convex function such
that
$\varphi(X)\in\mathcal{L}^{q+}(\Omega,\mathcal{F},\mathcal{P})$, we
have
\begin{eqnarray}\label{jensen}
\varphi(Z(X))\leq Z(\varphi(X)).
\end{eqnarray}
ii) If $U\in\mathcal{L}^{\infty}(\Omega,\mathcal{F},\mathcal{P})$,
the following identity holds true:
\begin{eqnarray}\label{left inverse}
UZ(X)=Z(X\mathcal{T}_{-b}U).
\end{eqnarray}
In particular for $X=1$ one gets
\begin{eqnarray}\label{left inverse 1}
U=Z(\mathcal{T}_{-b}U)\quad\mbox{ ($Z$ is the left-inverse of
$\mathcal{T}_{-b}$)}.
\end{eqnarray}
iii) If $E_{\mathcal{Q}}[\cdot|\mathcal{F}_T^b]$ denotes the
conditional expectation given $\mathcal{F}_T^b$ under the measure
$Q$, we have
\begin{eqnarray}\label{projection}
Z(X)=Z(E_{\mathcal{Q}}[X|\mathcal{F}_T^b]).
\end{eqnarray}
\end{proposition}

\begin{proof}
We prove each part of the statement separately. \\
i) In order to prove (\ref{jensen}) it is sufficient to show that
$Z(1)=1$ and that $Z$ preserves the positivity. These facts together
with the representation of a convex as a supremum of affine
functions will provide the desired inequality. Now using
(\ref{riesz}) and the Girsanov theorem we get
\begin{eqnarray*}
E[Z(1)Y]&=&E[\mathcal{E}(b)\mathcal{T}_{-b}Y]\\
&=&E[Y].
\end{eqnarray*}
The previous identities hold for any $Y\in
\mathcal{L}^{q'}(\Omega,\mathcal{F},\mathcal{P})$ implying that
$Z(1)=1$. Now if $X\geq 0$ then for any
$Y\in\mathcal{L}^{q'}(\Omega,\mathcal{F},\mathcal{P})$ with $Y\geq
0$ we obtain
\begin{eqnarray*}
E[Z(X)Y]=E[X\mathcal{E}(b)\mathcal{T}_{-b}Y]\geq 0.
\end{eqnarray*}
Hence $Z(X)\geq 0$.\\
ii) The proof of (\ref{left inverse}) is obtained as follows
\begin{eqnarray*}
E[UZ(X)Y]&=&E[X\mathcal{E}(b)\mathcal{T}_{-b}(UY)]\\
&=&E[X\mathcal{E}(b)\mathcal{T}_{-b}U\mathcal{T}_{-b}Y]\\
&=&E[Z(X\mathcal{T}_{-b}U)Y].
\end{eqnarray*}
Since $Y$ is arbitrary in
$\mathcal{L}^{q'}(\Omega,\mathcal{F},\mathcal{P})$ the proof is
complete.\\
iii) Let $Y\in\mathcal{L}^{q'}(\Omega,\mathcal{F},\mathcal{P})$.
Then
\begin{eqnarray*}
E[Z(X)Y]&=&E[X\mathcal{E}(b)\mathcal{T}_{-b}Y]\\
&=&E_{\mathcal{Q}}[X\mathcal{T}_{-b}Y]\\
&=&E_{\mathcal{Q}}[E_{\mathcal{Q}}[X|\mathcal{F}_T^b]\mathcal{T}_{-b}Y]\\
&=&E[E_{\mathcal{Q}}[X|\mathcal{F}_T^b]\mathcal{E}(b)\mathcal{T}_{-b}Y]\\
&=&E[Z(E_{\mathcal{Q}}[X|\mathcal{F}_T^b])Y].
\end{eqnarray*}
In the third equality we used the fact that $\mathcal{T}_{-b}Y$ is
$\mathcal{F}_T^b$-measurable since $\mathcal{T}_{-b}$ transforms
functionals of $B$ into functionals of $B^b$.
\end{proof}

\begin{remark}\label{commutation}\quad\\
Equality (\ref{left inverse 1}) implies a non trivial property of
the operator $Z$:\\
If $\varphi:\mathbb{R}\to\mathbb{R}$ is a bounded function,
\begin{eqnarray*}
\varphi(Z(\mathcal{T}_{-b}U))&=&\varphi(U)\\
&=&Z(\mathcal{T}_{-b}\varphi(U))\\
&=&Z(\varphi(\mathcal{T}_{-b}U))
\end{eqnarray*}
i.e.
\begin{eqnarray*}
\varphi(Z(\mathcal{T}_{-b}U))=Z(\varphi(\mathcal{T}_{-b}U)).
\end{eqnarray*}
This means that on the set of bounded functionals of the Brownian
motion $B^b$ ($\mathcal{T}_{-b}U$ is an element of this kind) the
operator $Z$ commutes with any bounded nonlinear function. Outside
this set we
can only get (\ref{jensen}).\\
Note that when $\mathcal{F}=\mathcal{F}^b$ the set of bounded
functionals of the Brownian motion $B^b$ coincides with the whole
$\mathcal{L}^{\infty}(\Omega,\mathcal{F},\mathcal{P})$. In this case
equation (\ref{SDE}) possess a unique strong solution (see Section 4
below).
\end{remark}

The next proposition establishes that the operator $Z$ preserves
also the $\{\mathcal{F}_t\}_{0\leq t\leq T}$-adaptedness.

\begin{proposition}\label{adaptedness}\quad\\
Let $\{X_t\}_{0\leq t\leq T}$ be an $\{\mathcal{F}_t\}_{0\leq t\leq
T}$-adapted stochastic process such that
$X_t\in\mathcal{L}^{2+}(\Omega,\mathcal{F},\mathcal{P})$ for all
$t\in [0,T]$. Then the process $\{Z(X_t)\}_{0\leq t\leq T}$ is also
$\{\mathcal{F}_t\}_{0\leq t\leq T}$-adapted.
\end{proposition}

\begin{proof}
We recall that the set
\begin{eqnarray*}
\{\mathcal{E}(f),f\in\mathcal{L}^2([0,T];\mathbb{R}^n)\}
\end{eqnarray*}
is total in $\mathcal{L}^q(\Omega,\mathcal{F},\mathcal{P})$ for all
$q\geq 1$ and that for any
$Y\in\mathcal{L}^2(\Omega,\mathcal{F},\mathcal{P})$,
\begin{eqnarray*}
E[Y\mathcal{E}(f)]=\sum_{k\geq 0}\langle h_k,f^{\otimes
n}\rangle_{\mathcal{L}^2([0,T]^k;\mathbb{R}^{n^k})}
\end{eqnarray*}
where the sequence $\{h_k\}_{k\geq 0}$ represents the kernels of the
Wiener-It\^o chaos decomposition of $Y$ with respect to the Brownian
motion $B$ (see e.g. \cite{HOUZ}).\\
We have
\begin{eqnarray*}
E[Z(X_t)\mathcal{E}(f)]&=&E[X_t\mathcal{E}(b)\mathcal{T}_{-b}\mathcal{E}(f)]\\
&=&E[X_t\mathcal{E}(f+b)].
\end{eqnarray*}
Observe that, as we already mentioned in Section 2, condition
(\ref{novikov}) ensures that the process
$\{\mathcal{E}_t(b+f)\}_{0\leq t\leq T}$ is a continuous
($\{\mathcal{F}_t\}_{0\leq t\leq T}$, $\mathcal{P}$)-martingale.
Therefore
\begin{eqnarray*}
E[Z(X_t)\mathcal{E}(f)]&=&E[X_t\mathcal{E}(f+b)]\\
&=&E[X_t\mathcal{E}_t(f+b)]\\
&=&E[X_t\mathcal{E}_t(b)\mathcal{T}_{-b}\mathcal{E}_t(f)]\\
&=&E[X_t\mathcal{E}(b)\mathcal{T}_{-b}\mathcal{E}_t(f)]\\
&=&E_{\mathcal{Q}}[X_t\mathcal{T}_{-b}\mathcal{E}_t(f)]\\
&=&E_{\mathcal{Q}}[E_{\mathcal{Q}}[X_t|\mathcal{F}_t^b]\mathcal{T}_{-b}\mathcal{E}_t(f)]\\
&=&E_{\mathcal{Q}}[E_{\mathcal{Q}}[X_t|\mathcal{F}_t^b]\mathcal{T}_{-b}\mathcal{E}(f)]\\
&=&\sum_{k\geq 0}\langle h_k(\cdot,t),f^{\otimes
k}\rangle_{\mathcal{L}^2([0,T]^k;\mathbb{R}^{n^k})}
\end{eqnarray*}
where $\{h_k(\cdot,t)\}_{k\geq 0}$ are the kernels of
$E_{\mathcal{Q}}[X_t|\mathcal{F}_t^b]$ with respect to the Brownian
motion $B^b$. Observe that in the seventh equality we utilized the
martingale property of $\{\mathcal{T}_{-b}\mathcal{E}_t(f)\}_{0\leq
t\leq T}$ with respect to the filtration $\{\mathcal{F}^b_t\}_{0\leq
t\leq T}$ and the measure $\mathcal{Q}$ (the role of
$\mathcal{T}_{-b}\mathcal{E}(f)$ in
$\mathcal{L}^2(\Omega,\mathcal{F},\mathcal{Q})$ is the same as the
one of $\mathcal{E}(f)$ in
$\mathcal{L}^2(\Omega,\mathcal{F},\mathcal{P})$).\\
Comparing the first and the last terms of the above chain of
equalities we deduce that $\{h_k(\cdot,t)\}_{k\geq 0}$ are also the
kernels of $Z(X_t)$ with respect to the Brownian motion $B$.\\
Since $E_{\mathcal{Q}}[X_t|\mathcal{F}_t^b]$ is
$\{\mathcal{F}_t^b\}_{0\leq t\leq T}$-adapted we get (see
\cite{HOUZ} Lemma 2.5.3) that for each $k\geq 0$,
$h_k(t_1,...,t_n,t)=0$ if $t_i>t$ for some $i\in\{1,...,n\}$. This
last condition (recall that the $h_k$'s are also the kernels of
$Z(X_t)$ with respect to the Brownian motion $B$) implies that
$Z(X_t)$ is $\{\mathcal{F}_t\}_{0\leq t\leq T}$-adapted.
\end{proof}

\begin{remark}\quad\\
Looking through the proof of Proposition \ref{adaptedness} one can
understand how the operator $Z$ acts:\\
- Take $X\in\mathcal{L}^{2+}(\Omega,\mathcal{F},\mathcal{P})$\\
- Compute the conditional expectation
$E_{\mathcal{Q}}[X|\mathcal{F}_T^b]$\\
- Write the Wiener-It\^o chaos decomposition  of
$E_{\mathcal{Q}}[X|\mathcal{F}_T^b]$ as $\sum_{k\geq 0}I^{b}_k(h_k)$
where $I^b_k$ are multiple It\^o integrals with respect to the
Brownian motion $B^b$\\
- Take the kernels from above and write the Wiener-It\^o chaos
decomposition $\sum_{k\geq 0}I_k(h_k)$ where $I_k$ are now multiple
It\^o integrals with respect to the Brownian motion $B$\\
- The result is Z(X), i.e. $Z(X)=\sum_{k\geq 0}I_k(h_k)$.
\end{remark}

\section{Proofs of the main results}

To ease the notation we will assume that $x=0$ and denote the
processes $X_t^0,Y_t^0$ and $Z_t^0$, appearing in the statement of
the theorem, by $X_t,Y_t$ and $Z_t$, respectively. Observe that the
one dimensional processes $Z_t^i$ as defined in (\ref{def of Z})
correspond in the notation of the previous section to $Z(B_t^i)$.
Therefore all the properties that we proved in the previous section
for $Z(X)$ are valid also for $Z_t^i$. We will however continue to
use the symbol $Z(B_t^i)$ when we will use the results of the previous section.\\

\noindent\textbf{Proof of Theorem \ref{main theorem}}\\

\noindent i) Adaptedness follows immediately from Proposition
\ref{adaptedness}. Let us prove the continuity. By means of
(\ref{jensen}) we have for any $i\in\{1,...,n\}$ and $s,t\in [0,T]$
that
\begin{eqnarray*}
E[|Z(B_t^i)-Z(B_s^i)|^3]&=&E[|Z(B_t^i-B_s^i)|^3]\\
&\leq&E[Z(|B_t^i-B_s^i|^3)]\\
&=&E[|B_t^i-B_s^i|^3\mathcal{E}(b)]\\
&\leq&E[|B_t^i-B_s^i|^6]^{\frac{1}{2}}E[\mathcal{E}(b)^2]^{\frac{1}{2}}\\
&=&C (|t-s|^3)^{\frac{1}{2}}\\
&=& C|t-s|^{\frac{3}{2}}.
\end{eqnarray*}
From the Kolmogorov's continuity theorem we deduce that $Z_t^i$ has
a continuous modification which we will continue to denote by
$Z_t^i$.\\

\noindent ii) Assume that $\mathcal{F}^b_t=\mathcal{F}_t$ for all
$t\in [0,T]$. This means that any functional of the path of $B$ is
also a functional of the path of $B^b$ and hence from Remark
\ref{commutation} we obtain the identity
\begin{eqnarray}\label{commutation 2}
\varphi(Z(U))=Z(\varphi(U))
\end{eqnarray}
for any measurable and bounded function
$\varphi:\mathbb{R}\to\mathbb{R}$ and
$U\in\mathcal{L}^{\infty}(\Omega,\mathcal{F},\mathcal{P})$. Actually
these hypothesis can be relaxed: in fact, looking through the proof
of identity (\ref{left inverse}) one can easily see that
(\ref{commutation 2}) holds for instance for all
$U\in\mathcal{L}^{2}(\Omega,\mathcal{F},\mathcal{P})$ and
$\varphi:\mathbb{R}\to\mathbb{R}$ such that
$\varphi(U)\in\mathcal{L}^{2}(\Omega,\mathcal{F},\mathcal{P})$. In
our case $U=B_t^i$ and $\varphi=b_i$. Since we are assuming
condition (\ref{novikov}) on the function $b$ a simple application
of the inequality $e^x\geq 1+x$ provides the condition
$E[b_i^2(B_t)]<+\infty$ $t$-$a.e.$. (The null set where the previous
condition is not fulfilled will play no role since property
(\ref{commutation 2}) will be used under the Lebesgue-integral sign).\\
Now fix a function $f\in\mathcal{L}^2([0,T];\mathbb{R}^n)$; an
application of the It\^o formula gives
\begin{eqnarray*}
E[Z_t^i\mathcal{E}(f)]&=&E[B_t^i\mathcal{E}(b)\mathcal{T}_{-b}\mathcal{E}(f)]\\
&=&E[B_t^i\mathcal{E}(f+b)]\\
&=&E\Big[B_t^i\Big(1+\int_0^T\langle
(f(s)+b(B_s))\mathcal{E}_s(f+b),dB_s\rangle\Big)\Big]\\
&=&E\Big[\int_0^t
(f_i(s)+b_i(B_s))\mathcal{E}_s(f+b)ds\Big]\\
&=&\int_0^tE[b_i(B_s)\mathcal{E}_s(f+b)]ds+\int_0^tf_i(s)ds\\
&=&\int_0^tE[b_i(B_s)\mathcal{E}(f+b)]ds+E[B_t^i\mathcal{E}(f)]\\
&=&\int_0^tE[b_i(B_s)\mathcal{E}(b)\mathcal{T}_{-b}\mathcal{E}(f)]ds+E[B_t^i\mathcal{E}(f)]\\
&=&\int_0^tE[Z(b_i(B_s))\mathcal{E}(f)]ds+E[B_t^i\mathcal{E}(f)]\\
&=&\int_0^tE[b_i(Z(B_s))\mathcal{E}(f)]ds+E[B_t^i\mathcal{E}(f)]\\
&=&E\Big[\Big(\int_0^tb_i(Z(B_s))ds+B_t^i\Big)\mathcal{E}(f)\Big].
\end{eqnarray*}
Comparing the first and the last term of this chain of equalities we
get for each $i\in\{1,...,n\}$ that
\begin{eqnarray*}
Z_t^i=\int_0^tb_i(Z_s)ds+B_t^i.
\end{eqnarray*}
This shows that the process $Z_t=(Z_t^1,...,Z^n_t)$ solves
(\ref{SDE}) in the strong sense. To prove uniqueness we simply
observe that if $\{V_t\}_{0\leq t\leq T}$ is another square
integrable, continuous, $\{\mathcal{F}_t\}_{0\leq t\leq T}$-adapted
solution to (\ref{SDE}) by applying the Girsanov theorem we get for
every $f\in\mathcal{L}^2([0,T];\mathbb{R}^n)$ and $i\in\{1...,n\}$
that
\begin{eqnarray*}
E[V^i_t\mathcal{E}(f)]=E[B_t^i\mathcal{E}(f+b)];
\end{eqnarray*}
since by definition
\begin{eqnarray*}
E[Z^i_t\mathcal{E}(f)]&=&E[B_t^i\mathcal{E}(b)\mathcal{T}_{-b}\mathcal{E}(f)]\\
&=&E[B_t^i\mathcal{E}(f+b)]
\end{eqnarray*}
we conclude that
\begin{eqnarray*}
Z_t=V_t\mbox{ for all }t\in [0,T]
\end{eqnarray*}
with probability one.\\

\noindent iii) Let us now assume that
$\mathcal{F}^b_t\subset\mathcal{F}_t$ for some $t\in [0,T]$. We can
repeat the reasoning that brought to the chain of equalities
obtained above; however now we have to stop before the interchange
between the action of the operator $Z$ and the nonlinear function
$b_i$ (in the present case this commutation is not allowed since the
Brownian motion $B$ can not be written as a functional of $B^b$).
Therefore we can write
\begin{eqnarray*}
E[Z_t^i\mathcal{E}(f)]&=&\int_0^tE[Z(b_i(B_s))\mathcal{E}(f)]ds+E[B_t^i\mathcal{E}(f)]\\
&=&E\Big[\Big(\int_0^tZ(b_i(B_s))ds+B_t^i\Big)\mathcal{E}(f)\Big]
\end{eqnarray*}
or equivalently
\begin{eqnarray}\label{before commutation}
Z_t^i=\int_0^tZ(b_i(B_s))ds+B_t^i.
\end{eqnarray}
Recall that in proving inequality (\ref{jensen}) we showed that $Z$
is a positivity preserving operator; due to the linearity of $Z$
this is equivalent to say that if $\mathcal{P}(X\leq Y)=1$ then
$\mathcal{P}(Z(X)\leq Z(Y))=1$. Therefore if we denote by
$\hat{b}_i$ the convex envelope of $b_i$ we can continue
(\ref{before commutation}) as
\begin{eqnarray*}
Z_t^i&=&\int_0^tZ(b_i(B_s))ds+B_t^i\\
&\geq&\int_0^tZ(\hat{b}_i(B_s))ds+B_t^i\\
&\geq&\int_0^t\hat{b}_i(Z(B_s))ds+B_t^i\\
&=&\int_0^t\hat{b}_i(Z_s)ds+B_t^i
\end{eqnarray*}
where in the second inequality we used (\ref{jensen}). Hence we
proved that $Z_t=(Z_t^1,...,Z_t^n)$ is a solution to the system of
stochastic differential inequalities
\begin{eqnarray*}
Z_t\geq\int_0^t\hat{b}(Z_s)ds+B_t.
\end{eqnarray*}
Now, the function $\hat{b}=(\hat{b}_1,...,\hat{b}_n)$ is a vector of
convex functions; since convex functions are locally
Lipschitz-continuous the stochastic differential equation
\begin{eqnarray*}
dY_t&=&\hat{b}(Y_t)dt+dB_t\quad t\in ]0,T]\\
Y_0&=&0
\end{eqnarray*}
possesses a unique strong solution up to the explosion time
\begin{eqnarray*}
\tau:=\lim_{N\to +\infty}\inf\{t\in[0,T]:\Vert Y_t\Vert>N\}
\end{eqnarray*}
(see e.g. \cite{IW}). If $\hat{b}:\mathbb{R}^n\to\mathbb{R}^n$
happens to be quasi-monotonously increasing then by Proposition 3.3
in \cite{AM} we conclude that
\begin{eqnarray*}
Z_t\geq Y_t\mbox{ for all }t\in [0,\tau]
\end{eqnarray*}
with probability one.\\

\noindent\textbf{Proof of Corollary \ref{corollary}}\\

\noindent First observe that for all $x\in\mathbb{R}^n$ we have
$b(x)\geq \hat{b}(x)\geq Ax+b$. Therefore following the line of
reasoning in the proof of $iii)$ above we can immediately say that
\begin{eqnarray*}
dZ_t\geq (AZ_t+b)dt+dB_t.
\end{eqnarray*}
Since the function $x\mapsto Ax+b$ is linear, the stochastic
differential equation
\begin{eqnarray*}
dY_t&=&(AY_t+b)dt+dB_t\\
Y_0&=&0
\end{eqnarray*}
possesses a unique global strong solution; this solution is
explicitly represented in the right hand side of (\ref{linear}) (see
\cite{KS} page 354). Moreover the assumption on $A$ guarantees that
the function $x\mapsto Ax+b$ is quasi-monotonously increasing. These
facts together with Proposition 3.3 in \cite{AM} implies the desired
inequality (\ref{linear}).

\end{document}